\documentclass[11pt]{amsart}
\usepackage[a4paper,left=2cm,right=2cm,bottom=2.5cm]{geometry}
\usepackage[utf8]{inputenc}
\usepackage[english]{babel}
\usepackage{amsmath,amssymb}
\usepackage{amsthm}
 \usepackage{stmaryrd} 
\usepackage{tikz}
\usepackage{tkz-tab}
\usepackage{listings}
\usepackage{graphicx}
\usepackage{caption}
\usepackage{url}
\usepackage[english]{babel}
\usepackage{amssymb}

\definecolor{mygreen}{RGB}{28,172,0} 
\definecolor{mylilas}{RGB}{170,55,241}

\newtheorem{theorem}{Theorem}[section]

\newtheorem{lemma}[theorem]{Lemma}
\newtheorem{proposition}[theorem]{Proposition}

\newtheorem{remark}[theorem]{Remark}

\begin{document}
\title{Where to place a spherical obstacle so as to maximize the first nonzero Steklov eigenvalue}
\author[]{Ilias Ftouhi}

\address[Ilias Ftouhi]{Friedrich-Alexander-Universität  Erlangen-Nürnberg, Department of Mathematics, Chair in Applied Analysis – Alexander von Humboldt Professorship, Cauerstr. 11, 91058 Erlangen, Germany.}
\email{ilias.ftouhi@fau.de}

\lstset{language=Matlab,%
    breaklines=true,%
    morekeywords={matlab2tikz},
    keywordstyle=\color{blue},%
    morekeywords=[2]{1}, keywordstyle=[2]{\color{black}},
    identifierstyle=\color{black},%
    stringstyle=\color{mylilas},
    commentstyle=\color{mygreen},%
    showstringspaces=false,
    numbers=left,%
    numberstyle={\tiny \color{black}},
    numbersep=9pt, 
    emph=[1]{for,end,break},emphstyle=[1]\color{red}, 
}

\date{December 30. 2021}

\begin{abstract}
We prove that among all doubly connected domains of $\mathbb{R}^n$ of the form $B_1\backslash \overline{B_2}$, where $B_1$ and $B_2$ are open balls of fixed radii such that $\overline{B_2}\subset B_1$, the first nonzero Steklov eigenvalue achieves its maximal value uniquely when the balls are concentric. Furthermore, we show that the ideas of our proof also apply to a mixed boundary conditions eigenvalue problem found in literature.
\end{abstract}



\maketitle


\section{Introduction}
\label{S:1}
\subsection{Optimization of the Steklov eigenvalue}
Let $\Omega\subset \mathbb{R}^n$ be a bounded open set with Lipschitz boundary. In this paper, we consider the following Steklov eigenvalue problem for the Laplace operator: 
\begin{equation}\label{steklov}
\begin{cases}
    \Delta u = 0 &\qquad\mbox{in $\Omega$,}\vspace{1mm} \\
    \ \frac{\partial u}{\partial n} = \sigma u & \qquad\mbox{on $\partial \Omega$,} 
\end{cases}
\end{equation}
where $\partial u /\partial n$ is the outer normal derivative of $u$ on $\partial \Omega$. It is {well-known} that the Steklov spectrum is discrete as long as the trace operator $H^1(\Omega)\rightarrow L^2(\partial \Omega)$ is compact, which is the case when the domain has Lipschitz boundary; in other words, in our framework the values of $\sigma$ for which the problem \eqref{steklov} admits nonzero solutions form an increasing sequence of eigenvalues $0=\sigma_0(\Omega)<\sigma_1(\Omega)\leq \sigma_2(\Omega)\leq \cdots \nearrow +\infty$, known as the Steklov spectrum of $\Omega$.\vspace{4px}

We are interested in the first nonzero Steklov eigenvalue, which can be given by a Rayleigh quotient:

$$\sigma_1(\Omega)=\inf\left\{ \frac{\int_\Omega |\nabla u|^2dx}{\int_{\partial \Omega} u^2 d\sigma}\ \ \Big|\ \ u\in H^1(\Omega)\backslash \{0\}\ \text{such that}\ \int_{\partial \Omega} u d \sigma= 0  \right\},$$
where the infimum is attained for the corresponding eigenfunctions.\vspace{5px} 


Among classical questions in spectral geometry, there are the problems of minimizing (or maximizing) the Laplace eigenvalues with various boundary conditions and different geometrical and topological constraints. The constraint of volume has been extensively studied  in the last years. For example, there is the celebrated Faber--Krahn inequality \cite{faber,krahn}, which states that the ball minimizes the first eigenvalue of the Laplacian with Dirichlet boundary condition among domains of fixed volume. There is a similar result for the maximization of the first nonzero eigenvalue of the Laplacian with Neumann boundary {condition} known as the Szeg\"{o}-Weinberger inequality \cite{MR61749,MR79286}. For the Steklov problem, F. Brock proved in \cite{brock} that the first nonzero eigenvalue of a lipschitz domain is less than the eigenvalue of the ball with the same volume.\vspace{5px} 


The perimeter constraint is very interesting to study, especially in the case of Steklov eigenvalues. One early result is due to Weinstock \cite{weinstock}, who used conformal mapping techniques to prove the following inequality for simply connected planar sets:  
$$P(\Omega)\sigma_1(\Omega)\leq P(B) \sigma_1(B),$$
where $P(\Omega)$ {represents} the perimeter of $\Omega$ and $B$ {is} a unit ball. 

Recently, A. Fraser and R. Schoen proved in \cite{fraser} that {for $n\ge 3$,} the ball does not maximize the first nonzero Steklov eigenvalue among all contractible domains of fixed boundary measure in $\mathbb{R}^n$. The proof was inspired from the following formula for the annulus:
$$P(B\backslash  \varepsilon B)^\frac{1}{n-1} \sigma_1(B\backslash  \varepsilon B) = P(B)^\frac{1}{n-1} \sigma_1(B) + \frac{1}{n-1}\varepsilon^{n-1}+ o(\varepsilon^{n-1})> P(B)^\frac{1}{n-1} \sigma_1(B),$$
where $\varepsilon B = \{\varepsilon x\ |\ x\in B\}$.

Note that by studying the variations of the function $\varepsilon \in [0,1]\longmapsto P(B\backslash  \varepsilon B)^\frac{1}{n-1} \sigma_1(B\backslash  \varepsilon B)$ one can prove that there exists a unique $\varepsilon_n\in (0,1)$ such that 
$$\forall \varepsilon \in [0,1),\ \ \ \ \ P(B\backslash  \varepsilon B)^\frac{1}{n-1} \sigma_1(B\backslash  \varepsilon B) \leq P(B\backslash  \varepsilon_n B)^\frac{1}{n-1} \sigma_1(B\backslash  \varepsilon_n B).$$

This motivates to look at the problem of maximizing $\sigma_1$ among domains with holes and wondering if the spherical shell $B\backslash  \varepsilon_n B$ maximizes $\sigma_1$ under perimeter constraint among some class of perforated domains, for example, the doubly connected ones. {Recently}, L. R.  Quinones used shape derivatives  to prove that the annulus $B \backslash \varepsilon_2 B$ is a critical shape of the first nonzero Steklov eigenvalue among planar doubly connected domains with fixed perimeter (see \cite{critical}).\vspace{5px}

In contrast with the result of \cite{fraser}, it was recently proved in \cite{bucur} that the Weinstock inequality is true in higher dimensions in the case of convex sets. Namely, the authors show that for every bounded convex set $\Omega \subset \mathbb{R}^n$, one has
$$P(\Omega)^\frac{1}{n-1}\sigma_1(\Omega)\leq P(B)^\frac{1}{n-1} \sigma_1(B).$$
{We also refer to \cite{MR4037463} for a quantitative version of the latter inequality for convex sets and to \cite{MR4216737} for some surprising stability and instability results in the case of planar simply connected sets.}\vspace{5px}

{
Enlightened with the discussion above, some natural questions arise: can we remove the topological constraints (convexity or simple connectedness) as for the Laplacian eigenvalues with other boundary conditions? Does there exist a domain which maximizes $\sigma_1$ under perimeter constraint? If not, can we determine the supremum of $\sigma_1$ on Lipschitz open sets? These questions have recently been completely settled for the planar case in \cite{new_lagace}, where the authors prove that the supremum of $\sigma_1(\Omega)P(\Omega)$ is given by $8\pi$ and that no maximizer exists. Moreover, they prove that any maximizing sequence $\Omega_m$ will have
an unbounded number of boundary components as $m$ goes to infinity, see \cite[Theorem 1.6 \& Corollary 1.7]{new_lagace}. As far as we know, such problems remain open in higher dimensions $n\ge 3$.}


\subsection{Perforated domains: state of the art}

The optimization of the placement of obstacles has interested many authors in the last decades. We briefly point out some classical and recent works in the topic.\vspace{5px}

Some early results, due to Payne and Weinberger \cite{aaa} on the one hand and Hersch \cite{bbb} on the other, are that for some extremum eigenvalue problems with mixed boundary conditions a certain annulus is the optimal set among multi-connected planar domains, i.e., whose boundary admits more than one component (see also \cite{ccc}). The main ideas consist in constructing judicious test functions by using the notion of web-functions (see \cite{gazolla} for more details on web functions). These ideas were very recently used and adapted for other similar problems (see \cite{indians,Paoli2019SharpEF}).
A classical family of obstacle problems that attracted a lot of attention was to find the best emplacement of a spherical hole inside a ball that optimizes the value of a given spectral functional (see  \cite{survey}, section (9)). An early result in this direction is that the first Dirichlet eigenvalue is maximal when the spherical obstacle is in the center of the larger ball. The proof is based on shape derivatives (see  \cite[Theorem 2.5.1]{henrot}) and on a reflection and domain monotonicity arguments, followed by the use of the boundary maximum principle. These arguments have been applied in greater generality by many authors: in \cite{MR1646670} Ramm and Shivakumar proved this result in dimension 2, in \cite{MR1983689} Kesavan gave a generalization to higher dimensions and showed a similar result for the Dirichlet energy, then Harrell, Kr\"{o}ger, and Kurata managed in \cite{MR1858877} to replace the exterior ball by a convex set which is symetric with respect to a given hyperplane.  In the same spirit, El Soufi and Kiwan proved in \cite{MR2410874} that the second Dirichlet eigenvalue is also maximal when the balls are concentric. Furthermore, many authors considered mixed boundary conditions problems, for instance in \cite{MR2126609}, while studying the internal stabilizability for a reaction–diffusion problem modeling a predator–prey system, the authors are led to consider an obstacle shape optimization problem for the first laplacian eigenvalue with mixed Dirichlet--Neumann boundary conditions. Another interesting work in the same direction is due to Bonder,  Groisman and Rossi, who studied the so called Sobolev trace inequality (see \cite{MR1978428,MR962929}), thus they were interested in the optimization of the first nonzero eigenvalue of an elliptic operator with mixed Dirichlet--Steklov boundary conditions among perforated domains: the existence and regularity of an optimal hole are proved in \cite{MR2408516,MR2193498}, and by using shape derivatives it is shown that annulus is a critical but not an optimum shape (see \cite{MR2408516}). At last, we point out the recent papers \cite{gavitone2021isoperimetric,ii,MR4191500,MR4227150,verma}, where the authors consider the first eigenvalue of the Laplace operator with mixed Dirichlet--Steklov boundary conditions. \vspace{5px}



Many examples stated in the last paragraph deal with linear operators eigenvalues in the special case of doubly connected domains with spherical outer and inner boundaries. The question we are treating in this paper belongs to this family of problems. Yet, it is also natural to seek for generalizations and the literature is quite rich of works treating more general cases: for results on linear operators with more general shapes of the domain and the obstacle in the euclidean case we refer to \cite{MR3470739,MR2368895,MR3787529,HenZuc17a,AFST_2018_6_27_4_863_0}, on the other hand, many results for manifolds were obtained by Anisa and Aithal \cite{MR2120602} in the setting of space-forms (complete simply connected Riemannian manifolds of constant sectional curvature), by Anisa and Vemuri \cite{MR3128768} in the setting of rank 1 symmetric spaces of non-compact type and by Aithal and Raut \cite{MR2945095} in the case of punctured regular polygons in two dimensional space forms. As for the case of non-linear operators we refer to the interesting progress made for the $p$-Laplace operator (see \cite{MR3841846,MR3427602}).



\subsection{Results of the paper}

In this paper, we are interested in finding the optimal placement of a spherical obstacle in a given ball so that the first nonzero Steklov eigenvalue is maximal. \vspace{3px}

Our main result is stated as follows: 

\begin{theorem}\label{main}
Among all doubly connected domains of $\mathbb{R}^n$ ($n\ge2$) of the form $B_1\backslash \overline{B_2}$, where $B_1$ and $B_2$ are open balls of fixed radii such that $\overline{B_2}\subset B_1$, the first nonzero Steklov eigenvalue achieves its maximal value uniquely when the balls are concentric.
\end{theorem}

In \cite{verma}, the authors consider a mixed Dirichlet--Steklov eigenvalue problem. They prove that  the first nonzero eigenvalue is maximal when the balls are concentric in dimensions larger or equal than $3$ (see \cite[Theorem 1]{verma}) and remark that the planar case remains open (see \cite[Remark 2]{verma}). We show that the ideas developed in this paper allow us to give an alternative and simpler proof of  \cite[Theorem 1]{verma}. Then we extend this result to the planar case. 
\begin{theorem}\label{main1}
Among all doubly connected domains of $\mathbb{R}^n$ ($n\ge2$) of the form $B_1\backslash \overline{B_2}$, where $B_1$ and $B_2$ are open balls of fixed radii such that $\overline{B_2}\subset B_1$, the first nonzero eigenvalue of the problem
$$\begin{cases}
    \Delta u = 0 &\qquad\mbox{in $B_1\backslash \overline{B_2}$,}\vspace{1mm} \\
    \ \ \ u = 0& \qquad\mbox{on $\partial B_2$,} \vspace{1mm} \\
    \   \frac{\partial u}{\partial n} = \tau u& \qquad\mbox{on $\partial B_1$.}
\end{cases}$$
 achieves its maximal value uniquely when the balls are concentric.
\end{theorem}

This paper is organized in 3 parts. First, we give the proof of Theorem \ref{main}. Then we use the ideas developed in Section \ref{s} to give a new proof of \cite[Theorem 1]{verma} and tackle the planar case which was up to our knowledge still open. Finally, Section \ref{s:appendix} is devoted to the computation of the Steklov eigenvalues and eigenfunctions of the spherical shell and the determination of the first nonzero Steklov eigenvalue (Theorem \ref{steklov_eigenvalue}) via a monotonicity result (Lemma \ref{monotonicity}).




\section{Proof of Theorem \ref{main}}\label{s}
By invariance with respect to rotations and translations and scaling properties of $\sigma_1$, we can reformulate the problem as follows: 

We assume that the obstacle $B_2$ is the open ball of radius $a\in(0,1)$ centred at the origin $O$ and $B_1 = y_d + B$, where $B$ is the unit ball centred at the origin $O$, $y_d := (0,...,0,d)\in \mathbb{R}^n$ and $d\in[0,1-a)$. What is the value of $d$ such that $\sigma_1(B_1\backslash \overline{B_2})$ is maximal?\vspace{5px}  

For every $d\in[0,1-a)$, we take $\Omega_d:= (y_d + B)\backslash a B $ (see Figure \ref{domaine}). \vspace{3px}

It is sufficient to prove that: 
$$\forall d\in(0,1-a),\ \ \ \ \ \ \sigma_1(\Omega_0)>\sigma_1(\Omega_d).$$

The proof is based on the following Proposition:

\begin{proposition}\label{prop_main}
There exists a function $f_n\in H^1(\mathbb{R}^n\backslash \overline{B_2})$  satisfying: 
\begin{enumerate}
    \item $f_n$ is an eigenfunction associated to $\sigma_1(\Omega_0)$ and can be used as a test function in the variational definition of $\sigma_1(\Omega_d)$. \vspace{4px}
    \item $\int_{\Omega_d} |\nabla f_n|^2dx \leq \int_{\Omega_0} |\nabla f_n|^2dx$, with equality if and only if $d=0$. \vspace{4px}
    \item $\int_{\partial\Omega_d} f_n^2d\sigma \ge \int_{\partial\Omega_0} f_n^2d\sigma$, with equality if and only if $d=0$. 
\end{enumerate}
\end{proposition}

Using Proposition \ref{prop_main}, we conclude as follows: 

$$\forall d\in(0,1-d),\ \ \ \ \ \sigma_1(\Omega_d)\leq\frac{\int_{\Omega_d} |\nabla f_n|^2dx}{\int_{\partial\Omega_d} f_n^2d\sigma}<\frac{\int_{\Omega_0} |\nabla f_n|^2dx}{\int_{\partial\Omega_0} f_n^2d\sigma}=\sigma_1(\Omega_0).$$ 

This proves Theorem \ref{main}.

\subsection{Proof of the first assertion of Proposition \ref{prop_main}}\label{first}
The first eigenvalue of the spherical shell $\Omega_0$ is computed in Theorem \ref{steklov_eigenvalue}. It is also proven that its multiplicity is equal to $n$ and the corresponding eigenfunctions are 
$$\begin{array}{ccccc}
u^i_n & : & \mathbb{R}^n & \longrightarrow & \mathbb{R} \\
 & & x=(x_1,\cdots,x_n) & \longmapsto & x_i\left(1+\frac{\mu_{\sigma,n}}{|x|^n}\right), \\
\end{array}$$
where $i\in \llbracket 1,n\rrbracket$ and $\mu_{\sigma,n} = \frac{1-\sigma_1(\Omega_0)}{n+\sigma_1(\Omega_0)-1}$.\vspace{4px}

Take $i\in \llbracket 1,n-1\rrbracket$. Since $\Omega_d$ is symmetrical to the hyperplane $\{x_i = 0\}$, we have
\begin{align*}
\int_{\partial \Omega_d}u^i_n d\sigma &\ \ =\ \int_{\partial \Omega_d\cap \{x_i \ge 0\}}u^i_n d\sigma + \int_{\partial \Omega_d\cap \{x_i \leq 0\}}u^i_n d\sigma   \\
&\ \ =\ \int_{\partial \Omega_d\cap \{x_i \ge 0\}}u^i_n d\sigma - \int_{\partial \Omega_d\cap \{x_i \ge 0\}}u^i_n d\sigma \ \ \ \ \ \text{(because $u_n^i(x_1,\dots,-x_i,\cdots,x_n)=-u_n^i(x_1,\dots,x_i,\dots,x_n)$)} \\
&\ \ =\  0.
\end{align*}

Thus, every eigenfunction $u_n^i$ (where $i\in \llbracket 1,n-1\rrbracket$) can be taken as a test function in the variational definition of $\sigma_1(\Omega_d)$ (note that this is not the case for $u_n^n$). This proves the first assertion of Proposition \ref{prop_main}.

\subsection{Spherical coordinates and preliminary computations}
Since the shapes considered are described by spheres, it is more convenient to work with the spherical coordinates instead of the Cartesian ones. \vspace{5px}

We set
$$\left\{
\begin{array}{l}
  x_1 = r\sin{\theta_1}\sin{\theta_1}\dots \sin{\theta_{n-2}}\sin{\theta_{n-1}},\vspace{2px}\\
  x_2 = r\sin{\theta_1}\sin{\theta_1}\dots \sin{\theta_{n-2}}\cos{\theta_{n-1}},\vspace{2px}\\
  \ \ \ \ \ \ \ \ \vdots\\
  x_{n-1} = r\sin{\theta_1}\cos{\theta_2},\vspace{2px}\\
  x_{n} = r\cos{\theta_1},
\end{array}
\right.$$
where $(r,\theta_1,\cdots,\theta_{n-1})\in \mathbb{R}^+\times[0,\pi]\times...\times[0,\pi]\times[0,2\pi]$.\vspace{5px}

Since, every eigenfunction $u_n^i$ (where $i\in \llbracket 1,n-1\rrbracket$) can be used as a test function in the variational definition of $\sigma_1(\Omega_d)$, we chose to take $f_n = u_n^{n-1}$ (see Remark \ref{why_test}). \vspace{5px}

Using spherical coordinates, we write
$$\begin{array}{ccccc}
f_2 & : & \mathbb{R}_+\times [0,2\pi] & \longrightarrow & \mathbb{R} \\
 & & (r,\theta_1) & \longmapsto & \sin{\theta_1}\left(r+\frac{\mu_{\sigma,n}}{r}\right), \\
\end{array}$$
and for $n\ge 3$
$$\begin{array}{ccccc}
f_n & : & \mathbb{R}_+\times[0,\pi]\times\cdots\times [0,\pi]\times [0,2\pi] & \longrightarrow & \mathbb{R} \\
 & & (r,\theta_1,\dots,\theta_{n-1}) & \longmapsto & \sin{\theta_1}\cos{\theta_2}\left(r+\frac{\mu_{\sigma,n}}{r^{n-1}}\right). \\
\end{array}$$

\begin{remark}\label{why_test}
The choice of the test function $f_n=u^{n-1}_n$ between all $u^i_n$ ($i\in \llbracket 1,n-1\rrbracket$) is motivated by the will to have less variables to deal with while computing the gradient (see Section \ref{gradient}). Nevertheless, one should note that all these functions satisfy the three assertions of Proposition \ref{prop_main}.\vspace{4px} 
\end{remark}

The following figure shows the perforated domains $\Omega_0$ and $\Omega_d$, the angle $\theta_1$ and the radius $R_d(\theta_1)$ which plays an important role in the upcoming computations. 

\begin{center}
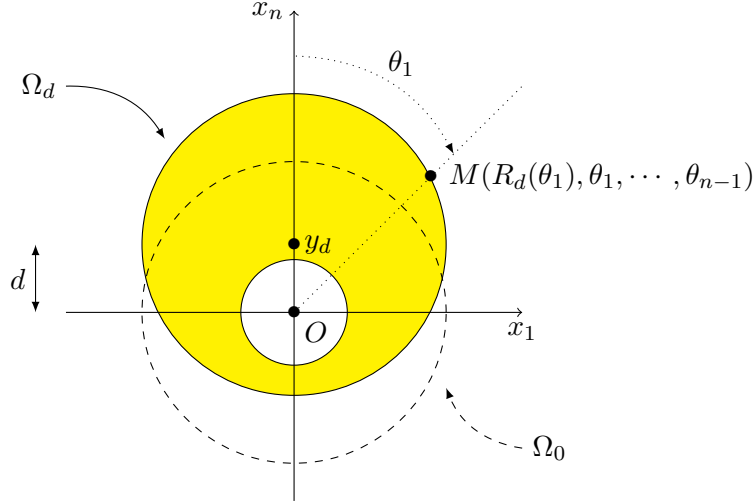
\label{domaine}
\begin{tikzpicture}

\draw[fill=yellow] (0,0.9) circle (2);
\draw[fill=white] (0,0) circle (.7);
\draw[dashed] (0,0) circle (2);
\draw (0,0) node {$\bullet$};
\draw (0,0.9) node {$\bullet$};
\draw (0,0.9) node[right] {$y_d$};
\draw[<->,>=latex]  (-3.4,0) to (-3.4,.9);
\draw (-3.4,0.45) node[left] {$d$};
\draw (0,0) node[below right] {$O$};
\draw (3,0) node[below] {$x_1$};
\draw (0,4) node[left] {$x_{n}$};
\draw[->] (-3,0)--(3,0);
\draw[dotted] (0,0)--(3,3);
\draw[->] (0,-2.5)--(0,4);
\draw (1.8,1.8) node {$\bullet$};
\draw (1.9,1.8) node[right] {$M(R_d(\theta_1),\theta_1,\cdots,\theta_{n-1})$};
\draw[<-,>=latex] [dotted] (2.1,2.1) to[bend right] (0,3.4);
\draw[->,>=latex]  (-3,3) to[bend left] (-1.7,2.3);
\draw[->,>=latex] [dashed] (3,-1.8) to[bend left] (2,-1);
\draw (3,-1.8) node[right] {$\Omega_0$};
\draw (-3,3) node[left] {$\Omega_d$};
\draw (1.1,3.3) node[right] {$\theta_1$};
\end{tikzpicture}
\captionof{figure}{The domains $\Omega_d$ and $\Omega_0$}
\end{center}

Let $M\in\partial(y_d+B)$, by using Al-Kashi's formula on the triangle $Oy_dM$,  we have
$$1^2=d^2+R_d^2(\theta_1)-2d R_d(\theta_1) \cos{\theta_1}.$$

By solving the equation of second degree satisfied by $R_d(\theta_1)$ we get two roots $d\cos{\theta_1} \pm \sqrt{1-d^2\sin^2 \theta_1}$. Since the smallest one is negative (due to the fact that $d\in [0,1)$), we deduce that 
$$ R_d(\theta_1) =  d\cos{\theta_1} + \sqrt{1-d^2\sin^2 \theta_1}.$$

We compute the first derivative of $R_d$, which appears in the area element when integrating on $\partial \Omega_d$ (more precisely on $\partial(y_d + B)$). 
$$  {R_d'}(\theta_1) = -d\sin{\theta_1} -\frac{d^2\sin{\theta_1}\cos{\theta_1}}{\sqrt{1-d^2\sin^2\theta_1}}=-\frac{d\sin{\theta_1}}{\sqrt{1-d^2\sin^2\theta_1}}\big(d\cos{\theta_1} + \sqrt{1-d^2\sin^2 \theta_1}\big). $$



With straightforward computations, we get the important equalities: 
\begin{equation}\label{equa}
\sqrt{R_d^2(\theta_1)+{R_d'}^2(\theta_1)}=1+\frac{d\cos{\theta_1}}{\sqrt{1-d^2\sin^2\theta_1}}=\frac{R_d(\theta_1)}{\sqrt{1-d^2\sin^2\theta_1}}\cdot
\end{equation}

\subsection{Proof of the second assertion of Proposition \ref{prop_main}}\label{gradient}
We compute the gradient of $f_n$ in the spherical coordinates and calculate the $L^2$-norm of its gradient on $\Omega_d$.\\ 

For $n=2$, we have
\[
\nabla f_2(r,\theta_1)
=
\begin{bmatrix}
    \frac{\partial f}{\partial r} 
    \\
    \frac{1}{r}\frac{\partial f}{\partial \theta_1}   
\end{bmatrix}
=
\begin{bmatrix}
    \sin{\theta_1}\left(1-\frac{\mu_{\sigma,n}}{r^2}\right)
    \\
    \cos{\theta_1}\left(1+\frac{\mu_{\sigma,n}}{r^2}\right)
\end{bmatrix},
\]
thus
\begin{align*}
\int_{\Omega_d} |\nabla f_2|^2 d x &\ \ =\ \int_{\theta_1=0}^{2\pi} \int_{r=a}^{R_d(\theta_1)} \left[\sin^2\theta_1\left(1-\frac{\mu_{\sigma,2}}{r^2}\right)^2 + {\cos^2\theta_1\left(1+\frac{\mu_{\sigma,2}}{r^2}\right)^2}\right]rdrd\theta_1 \\
&\ \ =\ \int_{\theta_1=0}^{2\pi} \int_{r=a}^{R_d(\theta_1)} \left[r+2\mu_{\sigma,2}\left(\cos^2\theta_1-\sin^2\theta_1\right)\frac{1}{r}+\frac{\mu_{\sigma,2}^2}{r^3}\right]dr d\theta_1\\
&\ \ =\ \int_{\theta_1=0}^{2\pi} \left(\frac{R_d^2(\theta_1)-a^2}{2}{+2\mu_{\sigma,2} \left(\cos^2\theta_1-\sin^2\theta_1\right)\ln{\left(\frac{R_d(\theta_1)}{a}\right)}}-\frac{\mu_{\sigma,2}^2}{2}\left(\frac{1}{R_d^2(\theta_1)}-\frac{1}{a^2}\right)\right)d\theta_1.
\end{align*}

In the same spirit, for $n\ge 3$, we have
\[
\nabla f_n(r,\theta_1,...,\theta_{n-1})
=
\begin{bmatrix}
    \frac{\partial f}{\partial r} 
    \\
    \frac{1}{r}\frac{\partial f}{\partial \theta_1}   
    \\
    \frac{1}{r\sin{\theta_1}}\frac{\partial f}{\partial \theta_2} 
    \\
     \frac{1}{r\sin{\theta_1}\sin{\theta_2}}\frac{\partial f}{\partial \theta_3} 
    \\
    \vdots 
    \\
 \frac{1}{r\sin{\theta_1}...\sin{\theta_{n-2}}}\frac{\partial f}{\partial \theta_{n-1}}
\end{bmatrix}
=
\begin{bmatrix}
    \sin{\theta_1}\cos{\theta_2}\left(1-\frac{(n-1)\mu_{\sigma,n}}{r^n}\right)
    \\
    \cos{\theta_1}\cos{\theta_2}\left(1+\frac{\mu_{\sigma,n}}{r^n}\right)
    \\
    -\sin{\theta_2} \left(1+\frac{\mu_{\sigma,n}}{r^n}\right)
    \\
    0
    \\
    \vdots 

    \\
    0
\end{bmatrix}.
\]

For $p\in \mathbb{N}$, we introduce $I_p:=\int_0^\pi\sin^pt dt$, which is the double of the classical Wallis integral. These integrals satisfy the essential recursive property 
\begin{equation}\label{recursive}
\forall p \in \mathbb{N},\ \ \ \ I_{p+2}=\frac{p+1}{p+2}I_p.
\end{equation}
{
Moreover, since $I_0=\pi$ and $I_1= 2$, it is classical to use the recursive property \eqref{recursive} to compute the values of all the other integrals. We have
\begin{equation}\label{wallis_values}
\forall k\in\mathbb{N},\ \ \ \ \ \ \ \left\{\begin{matrix}  I_{2k}=\pi \frac{(2k)!}{\left(2^k k!\right)^2}=\frac{\pi}{2^{2k}}\binom{2k}{k},
\vspace{5px}
\\
I_{2k+1} = 2\cdot \frac{\left(2^k k!\right)^2}{(2k+1)!}.
\end{matrix}\right.
\end{equation}
}

We compute \smaller\smaller
\begin{align*}
A_1^n(d) &\ \ =\ \int_{{\Omega_d}} [\nabla f_n]_1^2 dx\\
&\ \ =\ 2\int_{\theta_1=0}^\pi...\int_{\theta_{n-1}=0}^{\pi}\int_{r=a}^{R_d(\theta_1)} \sin^2{\theta_1}\cos^2{\theta_{2}}\left(1-\frac{(n-1)\mu_{\sigma,n}}{r^n}\right)^2r^{n-1} \prod_{i=1}^{n-2}\sin^{n-1-i}\theta_idr d\theta_1...d\theta_{n-1}\\
&\ \ =\ 2\left(\prod_{k=0}^{n-4}I_k\right) \int_{\theta_{2}=0}^\pi \cos^2\theta_2\sin^{n-3}\theta_2 d\theta_{2} \int_{\theta_1=0}^\pi \sin^n\theta_1  \int_{r=a}^{R_d(\theta_1)}\left(r^{n-1}-\frac{2(n-1)\mu_{\sigma,n}}{r}+\frac{(n-1)^2\mu_{\sigma,n}^2}{r^{n+1}}\right)dr d\theta_1 \\
&\ \ =\ 2\left(\prod_{k=0}^{n-4}I_k\right)(I_{n-3}-I_{n-1})  \int_{\theta_1=0}^\pi \sin^n\theta_1 \left(\frac{R_d^n(\theta_1)-a^n}{n}-2(n-1)\mu_{\sigma,n} \ln{\left(\frac{R_d(\theta_1)}{a}\right)}-\frac{(n-1)^2\mu_{\sigma,n}^2}{n}\left( \frac{1}{R_d^n(\theta_1)}-\frac{1}{a^n}\right)\right)d\theta_1 \\
&\ \ =\ \frac{2}{n-1}\left(\prod_{k=0}^{n-3}I_k\right) \int_{\theta_1=0}^\pi \sin^n\theta_1 \left(\frac{R_d^n(\theta_1)-a^n}{n}-2(n-1)\mu_{\sigma,n} \ln{\left(\frac{R_d(\theta_1)}{a}\right)}-\frac{(n-1)^2\mu_{\sigma,n}^2}{n}\left( \frac{1}{R_d^n(\theta_1)}-\frac{1}{a^n}\right)\right)d\theta_1, 
\end{align*}\larger \larger
where we used \eqref{recursive} for the last equality. \smaller\smaller
\begin{align*}
A_2^n(d) &\ \ =\ \int_{{\Omega_d}} [\nabla f_n]_2^2 d x\\
&\ \ =\ 2\left(\prod_{k=0}^{n-4}I_k\right)(I_{n-3}-I_{n-1}) \int_{\theta_1=0}^\pi \cos^2\theta_1 \sin^{n-2}\theta_1 \left(\frac{R_d^n(\theta_1)-a^n}{n}+2\mu_{\sigma,n} \ln{\left(\frac{R_d(\theta_1)}{a}\right)}-\frac{\mu_{\sigma,n}^2}{n}\left( \frac{1}{R_d^n(\theta_1)}-\frac{1}{a^n}\right)\right)d\theta_1\\
&\ \ =\ \frac{2}{n-1}\left(\prod_{k=0}^{n-3}I_k\right) \int_{\theta_1=0}^\pi  (\sin^{n-2}\theta_1-\sin^{n}\theta_1) \left(\frac{R_d^n(\theta_1)-a^n}{n}+2\mu_{\sigma,n} \ln{\left(\frac{R_d(\theta_1)}{a}\right)}-\frac{\mu_{\sigma,n}^2}{n}\left( \frac{1}{R_d^n(\theta_1)}-\frac{1}{a^n}\right)\right)d\theta_1,
\end{align*}
\larger\larger
then
\begin{align*}
A_3^n(d) &\ \ =\ \int_{{\Omega_d}} [\nabla f_n]_3^2 d x\\ 
&\ \ =\ 2\left(\prod_{k=0}^{n-4}I_k\right)I_{n-1} \int_{\theta_1=0}^\pi \sin^{n-2}\theta_1 \left(\frac{R_d^n(\theta_1)-a^n}{n}+2\mu_{\sigma,n} \ln{\left(\frac{R_d(\theta_1)}{a}\right)}-\frac{\mu_{\sigma,n}^2}{n}\left( \frac{1}{R_d^n(\theta_1)}-\frac{1}{a^n}\right)\right)d\theta_1\\
&\ \ =\ \frac{2(n-2)}{n-1}\left(\prod_{k=0}^{n-3}I_k\right) \int_{\theta_1=0}^\pi \sin^{n-2}\theta_1 \left(\frac{R_d^n(\theta_1)-a^n}{n}+2\mu_{\sigma,n} \ln{\left(\frac{R_d(\theta_1)}{a}\right)}-\frac{\mu_{\sigma,n}^2}{n}\left( \frac{1}{R_d^n(\theta_1)}-\frac{1}{a^n}\right)\right)d\theta_1.
\end{align*}

We decompose the integral in three parts:{
\begin{equation}\label{equality}
\int_{\Omega_d}|\nabla f|^2 d x = A_1^n(d)+A_2^n(d)+A_3^n(d) = \frac{2}{n-1}\left(\prod_{k=0}^{n-3}I_k\right)\left( \frac{n-1}{n}W_1^n(d)+2\mu_{\sigma,n} W_2^n(d) -\frac{\mu_{\sigma,n}^2}{n}W_3^n(d)\right),
\end{equation}}
where
$$\left\{
\begin{array}{l}
  W_1^n(d) = \int_{0}^\pi \sin^{n-2}\theta_1 \big(R_d^n(\theta_1)-a^n\big)d\theta_1,\\
  \\
  W_2^n(d) = \int_{0}^\pi \phi_n(\theta_1)\ln{\left(\frac{R_d(\theta_1)}{a}\right)}d\theta_1,\ \ \ \ \ \ \ \text{with $\phi_n(\theta_1) = -n\sin^n\theta_1 + (n-1)\sin^{n-2}\theta_1$, }    \\
  \\
  W_3^n(d) = \int_{0}^\pi \psi_n(\theta_1)\left(\frac{1}{R_d^n(\theta_1)}-\frac{1}{a^n}\right)d\theta_1,\ \ \ \text{with $\psi_n(\theta_1) =  n(n-2)\sin^n\theta_1 + (n-1)\sin^{n-2}\theta_1\ge 0$.}    \\
\end{array}
\right.$$

Note that the equality \eqref{equality} applies also for the planar case. From now on, we take $n\ge 2$.\vspace{5px}  

In the following Lemma, we study $W_k^n(d)$ for each $k\in\{1,2,3\}$. 

\begin{lemma}\label{l1}
For every $n\ge 2$ and every $d\in[0,1-a]$:
\begin{enumerate}
    \item $W_1^n(d) = W_1^n(0)$.\vspace{4px}
    \item $W_2^n(d) = 0$.\vspace{4px}
    \item $W_3^n(d) \ge W_3^n(0)$, with equality if and only of $d=0$.
\end{enumerate}
\end{lemma}
\begin{proof}
\begin{enumerate}
    \item The idea is to see that the quantities $W^n_1(0)$ and $W^n_1(d)$ can be interpreted (up to a multiplicative constant) as volumes of the unit balls $B$ and $y_d+B$ in $\mathbb{R}^n$. Then, since the measure is invariant by translations, we get the equality. We have
{
\begin{align*}
W_1^n(d) &\ \ =\ \int_{0}^\pi \sin^{n-2}\theta_1 \big(R_d^n(\theta_1)-a^n\big)d\theta_1\\ 
&\ \ =\ \frac{n}{\prod_{k=0}^{n-3}I_k}\int_{\theta_1=0}^\pi...\int_{\theta_{n-1}=0}^{\pi}\int_{r=a}^{R_d(\theta_1)}1\times r^{n-1} \prod_{i=1}^{n-2}\sin^{n-1-i}\theta_i dr d\theta_1...d\theta_{n-1} \\
&\ \ =\  \frac{n}{2\prod_{k=0}^{n-3}I_k} |\Omega_d|\\ 
&\ \ =\ \frac{n}{2\prod_{k=0}^{n-3}I_k} \left(|y_d+B|-|a B|\right)\\ 
&\ \ =\ \frac{n}{2\prod_{k=0}^{n-3}I_k} \left(|B|-|a B|\right)\\
&\ \ =\ \frac{n}{\prod_{k=0}^{n-3}I_k}\int_{\theta_1=0}^\pi...\int_{\theta_{n-1}=0}^{\pi}\int_{r=a}^{1}1\times r^{n-1} \prod_{i=1}^{n-2}\sin^{n-1-i}\theta_i  dr d\theta_1...d\theta_{n-1} \\
&\ \ =\ W_1^n(0).
\end{align*}}

\item 
We remark that for every $\theta_1 \in (0,\pi)$ one has $\phi_n(\pi-\theta_1)=\phi_n(\theta_1)$, thus
\begin{align*}
W_2(d) &\ \ =\ \int_{0}^{\frac{\pi}{2}} \phi_n(\theta_1)\ln{\left(\frac{R_d(\theta_1)}{a}\right)}d\theta_1 + \int_{\frac{\pi}{2}}^\pi \phi_n(\theta_1)\ln{\left(\frac{R_d(\theta_1)}{a}\right)}d\theta_1\\
&\ \ =\ \int_{0}^{\frac{\pi}{2}} \phi_n(\theta_1)\ln{\left(\frac{d\cos{\theta_1}+\sqrt{1-d^2\sin^2\theta_1}}{a}\right)}d\theta_1 + \int_{0}^{\frac{\pi}{2}} \phi_n(\theta_1)\ln{\left(\frac{-d\cos{\theta_1}+\sqrt{1-d^2\sin^2\theta_1}}{a}\right)}d\theta_1\\
&\ \ =\ \int_{0}^{\frac{\pi}{2}} \phi_n(\theta_1)\ln{\left(\frac{\big(\ d\cos{\theta_1}+\sqrt{1-d^2\sin^2\theta_1}\ \big)\times\big(\ -d\cos{\theta_1}+\sqrt{1-d^2\sin^2\theta_1}\ \big)}{a^2}\right)}d\theta_1\\
&\ \ =\ \ln{\left(\frac{1-d^2}{a^2}\right)}\times \int_{0}^{\frac{\pi}{2}} \phi_n(\theta_1)d\theta_1\\
&\ \ =\ \ln{\left(\frac{1-d^2}{a^2}\right)}\times \int_0^{\frac{\pi}{2}}(-n\sin^n\theta_1 + (n-1)\sin^{n-2}\theta_1)d\theta_1\\
&\ \ =\ \frac{1}{2}\ln{\left(\frac{1-d^2}{a^2}\right)}\times\Big(-n I_n +(n-1) I_{n-2}\Big) = 0\ \ \ \ \ \ \text{(by \eqref{recursive})}. 
\end{align*}

\item
We have
\begin{align*}
W_3(d) &\ \ =\  \int_{\theta_1=0}^\pi \psi_n(\theta_1)\left(\frac{1}{\big(\ d\cos{\theta_1}+\sqrt{1-d^2\sin^2\theta_1}\ \big)^n}-\frac{1}{a^n}\right)d\theta_1   \\
&\ \ \ge\ \int_{\theta_1=0}^\pi \psi_n(\theta_1)\left(\frac{1}{\big(d\cos{\theta_1}+1\big)^n}-\frac{1}{a^n}\right)d\theta_1 =: G(d)\ge G(0)=W_3(0).
\end{align*}
The inequality $G(d)\ge G(0)$ is a consequence of the monotonicity of the function $G$ and equality occurs if and only if $d=0$. Indeed for every $d\in(0,1-a)$:
\begin{align*}
G'(d) &\ \ =\  \int_{0}^\pi -n \psi_n(\theta_1)\frac{\cos{\theta_1}}{\big(d\cos{\theta_1}+1 \big)^{n+1}}d\theta_1   \\
&\ \ =\ \int_{0}^{\frac{\pi}{2}} -n \psi_n(\theta_1)\frac{\cos{\theta_1}}{\big(d\cos{\theta_1}+1 \big)^{n+1}}d\theta_1+\int_{\theta_1={\frac{\pi}{2}}}^{\pi} -n \psi_n(\theta_1)\frac{\cos{\theta_1}}{\big(d\cos{\theta_1}+1 \big)^{n+1}}d\theta_1\\
&\ \ =\ n \int_{0}^{\frac{\pi}{2}} \psi_n(\theta_1)\cos{\theta_1}\left(\frac{1}{(1-d\cos{\theta_1})^{n+1}}-\frac{1}{(1+d\cos{\theta_1})^{n+1}}\right)d\theta_1\\
&\ \ > \ 0 \ \ \ \ \ \ \text{(because $\forall \theta_1 \in (0,\pi/2),\ \ \ \ \psi_n(\theta_1)\cos{\theta_1}> 0 \ \ \ \text{and}\ \ \ \ (1+d\cos{\theta_1})^{n+1}> (1-d\cos{\theta_1})^{n+1}$).}
\end{align*}
\end{enumerate}

\end{proof}

Using the results of Lemma \ref{l1}, we get 
\begin{align*}
\int_{\Omega_d} |\nabla f|^2 d x  &\ \ =\  \frac{2}{n-1}\left(\prod_{k=0}^{n-3}I_k\right)\left( \frac{n-1}{n}W_1^n(d)+2\mu_{\sigma,n} W_2^n(d) -\frac{\mu_{\sigma,n}^2}{n}W_3^n(d)\right)   \\
&\ \ \leq\ \frac{2}{n-1}\left(\prod_{k=0}^{n-3}I_k\right)\left( \frac{n-1}{n}W_1^n(0)+2\mu_{\sigma,n} W_2^n(0) -\frac{\mu_{\sigma,n}^2}{n}W_3^n(0)\right)=\int_{\Omega_0} |\nabla f|^2 d x, 
\end{align*}
with equality if and only if $d=0$. This proves the second assertion of Proposition \ref{prop_main}.

\subsection{Proof of the third assertion of Proposition \ref{prop_main}}\vspace{2px}
Take $n\ge2$. We have\smaller\smaller
\begin{align*}
\int_{\partial (y_d+B)}f_n^2d\sigma &\ \ =\ 2\int_{\theta_1=0}^\pi...\int_{\theta_{n-1}=0}^{\pi} f_n^2(r,\theta_1,...,\theta_{n-1}) R^{n-2}_d(\theta_1)\prod_{i=1}^{n-2}\sin^{n-1-i}\theta_i \sqrt{R_d^2(\theta_1)+{R_d'}^2(\theta_1)}\ d\theta_1...d\theta_{n-1}\\
&\ \ =\ 2\int_{\theta_1=0}^\pi...\int_{\theta_{n-1}=0}^{\pi} \sin^2\theta_1 \cos^2\theta_2 \left(R_d(\theta_1)+\frac{\mu_{\sigma,n}}{R_d^{n-1}(\theta_1)}\right)^2 R^{n-2}_d(\theta_1)\prod_{i=1}^{n-2}\sin^{n-1-i}\theta_i \sqrt{R_d^2(\theta_1)+{R_d'}^2(\theta_1)}\ d\theta_1...d\theta_{n-1}\\
&\ \ =\ 2\int_{\theta_1=0}^\pi...\int_{\theta_{n-1}=0}^{\pi} \sin^n\theta_1  \left(R_d(\theta_1)+\frac{\mu_{\sigma,n}}{R_d^{n-1}(\theta_1)}\right)^2 R^{n-2}_d(\theta_1) \sqrt{R_d^2(\theta_1)+{R_d'}^2(\theta_1)}
\cos^2\theta_2\prod_{i=2}^{n-2}\sin^{n-1-i}\theta_i \ d\theta_1...d\theta_{n-1}\\
&\ \ =\ \left(2\prod_{k=2}^{n-1}I_k \right)\int_{0}^\pi \left(R_d(\theta_1)+\frac{\mu_{\sigma,n}}{R_d^{n-1}(\theta_1)}\right)^2 R^{n-2}_d(\theta_1)\sin^{n}{\theta_1}\sqrt{R_d^2(\theta_1)+{R_d'}^2(\theta_1)}d\theta_1\\
&\ \ =\ \left(2\prod_{k=2}^{n-1}I_k \right) \int_0^\pi \left(R_d^n(\theta_1)+2\mu_{\sigma,n} +\frac{\mu_{\sigma,n}^2}{R_d^n(\theta_1)}\right)\sin^n\theta_1 \sqrt{R_d^2(\theta_1)+{R_d'}^2(\theta_1)}\ d\theta_1 \\
&\ \ =\ \left(2\prod_{k=2}^{n-1}I_k \right)\big(V_1^n(d)+2\mu_{\sigma,n} (I_n+V_2^n(d))+\mu_{\sigma,n}^2V_3^n(d)\big),
\end{align*}\larger\larger
where 
$$\left\{
\begin{array}{l}
  V_1^n(d) = \int_{0}^\pi \sin^{n}\theta_1 R_d^n(\theta_1)\sqrt{R_d^2(\theta_1)+{{R'_d}}^2(\theta_1)} d\theta_1, \vspace{5px}
  \\
  V_2^n(d) = \int_{0}^\pi \sin^n\theta_1 \frac{d\cos{\theta_1}}{\sqrt{1-d^2\sin^2\theta_1}}d\theta_1\ \ \ \ (\text{by using  \eqref{equa}}), \vspace{5px}
  \\
  V_3^n(d) = \int_{0}^\pi \frac{\sin^n\theta_1}{R_d^{n-1}(\theta_1)\sqrt{1-d^2\sin^2\theta_1}}d\theta_1\ \ \ \ (\text{by using  \eqref{equa}}).   
\end{array}
\right.$$\\

Let us now prove the following Lemma:

\begin{lemma}\label{l2}
For every $n\ge1$ and every $d\in[0,1-a]$, we have:
\begin{enumerate}
    \item $V_1^n(d)=V_1^n(0)$.\vspace{4px}
    \item  $V_2^n(d)=0$.\vspace{4px}
    \item $V_3^n(d)\ge V_3^n(0)$, with equality if and only of $d=0$.
\end{enumerate}
\end{lemma}
\begin{proof}
\begin{enumerate}
    \item Take $B$ the unit ball of $\mathbb{R}^{n+2}$ centred at the origin $O$  and $y_d=(0,\cdots,0,d)\in \mathbb{R}^{n+2}$.
    
    In the same spirit of the proof of the assertion 1 of Lemma \ref{l1}, the idea is to see that the quantities $V^n_1(0)$ and $V^n_1(d)$ can be interpreted (up to a multiplicative constant) as the perimeters of $B$ and $y_d+B$. Then, since the perimeter is invariant by translations, we get the equality.\vspace{6px}
    
    We have
\begin{align*}
V_1^n(d) &\ \ =\  \int_{0}^\pi \sin^{n}\theta_1 R_d^n(\theta_1)\sqrt{R_d^2(\theta_1)+{{R'_d}}^2(\theta_1)} d\theta_1   \\
&\ \ =\ \frac{1}{2\prod_{k=2}^{n-1}I_k}\cdot 2\int_{\theta_1=0}^\pi...\int_{\theta_{n+1}=0}^{\pi}1\times R_d^{n}(\theta_1)\sqrt{R_d^2(\theta_1)+{{R'_d}}^2(\theta_1)} \prod_{i=1}^{n}\sin^{n+1-i}\theta_i d\theta_1...d\theta_{n+1}\\
&\ \ =\ \frac{1}{2\prod_{k=2}^{n-1}I_k} P(y_d+B)\\
&\ \ =\  \frac{1}{2\prod_{k=2}^{n-1}I_k} P(B)\\
&\ \ =\  V_1^n(0). 
\end{align*}

\item
We have
\begin{align*}
V_2^n(d) &\ \ =\ \int_{0}^{\frac{\pi}{2}} \sin^n\theta_1 \frac{d\cos{\theta_1}}{\sqrt{1-d^2\sin^2\theta_1}}d\theta_1+\int_{\frac{\pi}{2}}^\pi \sin^n\theta_1 \frac{d\cos{\theta_1}}{\sqrt{1-d^2\sin^2\theta_1}}d\theta_1\\
&\ \ =\ \int_{0}^{\frac{\pi}{2}} \sin^n\theta_1 \frac{d\cos{\theta_1}}{\sqrt{1-d^2\sin^2\theta_1}}d\theta_1-\int_0^{\frac{\pi}{2}} \sin^n t \frac{d\cos{t}}{\sqrt{1-d^2\sin^2t}} dt\\
&\ \ =\ 0.
\end{align*}

\item
We have
\begin{align*}
V_3^n(d) &\ \ =\ \int_{0}^\pi \frac{\sin^n\theta_1}{\big(d\cos{\theta_1}+\sqrt{1-d^2\sin^2\theta_1}\big)^{n-1}\sqrt{1-d^2\sin^2\theta_1}}d\theta_1\\
&\ \ \ge\ \int_{0}^\pi \frac{\sin^n\theta_1}{\big(d\cos{\theta_1}+1\big)^{n-1}}d\theta_1 =: H(d) \ge H(0)=V_3^n(0).
\end{align*}
The inequality $H(d)\ge H(0)$ follows from the monotonicity of $H$ and is an equality if and only if $d=0$. This can be proved with the same method used for $G$ in the previous section.
\end{enumerate}
\end{proof}

Using the results of Lemma \ref{l2}, we get 
\begin{align*}
\int_{\partial \Omega_d} f^2d\sigma  &\ \ =\ \left(2\prod_{k=2}^{n-1}I_k \right)\big(V_1^n(d)+2\mu_{\sigma,n} (I_n+V_2^n(d))+\mu_{\sigma,n}^2V_3^n(d)\big)+\int_{\partial (a B)} f^2d\sigma\\
&\ \ \ge\ \left(2\prod_{k=2}^{n-1}I_k \right)\big(V_1^n(0)+2\mu_{\sigma,n} (I_n+V_2^n(0))+\mu_{\sigma,n}^2V_3^n(0)\big)+\int_{\partial (a B)} f^2d\sigma\\
&\ \ =\ \int_{\partial \Omega_0} f^2d\sigma,
\end{align*}
which proofs the third assertion of Proposition \ref{prop_main}. 
\section{The Dirichlet--Steklov problem}
In this section, we show that the ideas of our proof in Section \ref{s} also apply to the problem considered in \cite{verma}. Thus, we give an alternative proof of \cite[Theorem 1]{verma} which deals with $n\ge3$ and tackle the planar case which {is} to our knowledge still open (see \cite[Remark 2]{verma}). \vspace{3px}

Let $n\ge2$ and $B_1$ be an open ball in $\mathbb{R}^n$ and $B_2$ be an open ball contained in $B_1$. We are interested in the eigenvalue problem 
$$\begin{cases}
    \Delta u = 0 &\qquad\mbox{in $B_1\backslash \overline{B_2}$,}\vspace{1mm} \\
    \ \ \ u = 0& \qquad\mbox{on $\partial B_2$,} \vspace{1mm} \\
    \   \frac{\partial u}{\partial n} = \tau u& \qquad\mbox{on $\partial B_1$.}
\end{cases}$$
The first eigenvalue of $B_1\backslash \overline{B_2}$ is given by the following Rayleigh quotient: 
$$\tau_1\left(B_1\backslash \overline{B_2}\right)=\inf\left\{ \frac{\int_{B_1\backslash \overline{B_2}} |\nabla u|^2dx}{\int_{\partial B_1} u^2 d\sigma}\ \ \Big|\ \ u\in H^1(\Omega)\backslash \{0\}\ \text{such that}\  u = 0\  \text{on $\partial B_2$} \right\}.$$
As stated in Theorem \ref{main1}, the eigenvalue $\tau_1$ is maximal when the balls are concentric. As for the case of pure Steklov boundary condition, we can assume without loss of generality that the obstacle $B_2$ is the open ball of radius $a\in(0,1)$ centred at the origin $O$ and $B_1 = y_d + B$, where $B$ is the unit ball centred at the origin $O$. We use the notations introduced in Section \ref{s}. \vspace{5px}

Using separation of variables S. Verma and G. Santhanam proved {in \cite[Section 2.1]{verma}} that the first  eigenfunction $g_n$ of the spherical shell $\Omega_0$ is given by: 
$$g_n(r,\theta_1,...,\theta_{n-1})=
\begin{cases}
    \ln{r}-\ln{a}, &\qquad\mbox{if $n=2$,}\vspace{1mm} \\
     \left(\frac{1}{a^{n-2}}-\frac{1}{r^{n-2}} \right ), & \qquad\mbox{if $n\ge 3$.}
\end{cases}$$
\subsection{A key proposition}

Here also, Theorem \ref{main1} is an immediate consequence of the following proposition: 
\begin{proposition}\label{propo}
Let $n\ge2$. We have: 
\begin{enumerate}
    \item $g_n$ can be used as a test function in the variational definition of $\tau_1(\Omega_d)$.\vspace{6px} 
    \item $\int_{\Omega_d} |\nabla g_n|^2dx \leq \int_{\Omega_0} |\nabla g_n|^2dx$. \vspace{4px}
    \item $\int_{\partial(y_d+B)} g_n^2d\sigma \ge \int_{\partial B} g_n^2d\sigma$, with equality if and only if $d=0$. 
\end{enumerate}
\end{proposition}
\begin{proof}
This proposition has been proved in \cite{verma} for the case $n\ge 3$.  \vspace{4px}

The first assertion is obvious since $g_n(a,\theta_1,\cdots,\theta_{n-1})=0$.  \vspace{4px}

As for the second, it has been remarked in \cite{verma}\vspace{4px} page 13, the inequality $\int_{\Omega_d} |\nabla g_n|^2 dx\leq \int_{\Omega_0} |\nabla g_n|^2 dx$ is a straightforward consequence of the monotonicity of $r\longmapsto \frac{\partial g_n}{\partial r}$. Unfortunately, this is not the case for the inequality on the boundary (assertion 3) for which the author needs more computations (see \cite[Section 2.2]{verma}). \vspace{4px}

First, we show that Lemma \ref{l2} allows us to give an alternative and simpler proof of the last inequality in the case $n\ge 3$, then we prove it in the planar case $n=2$. \vspace{5px} 

If $n\ge 3$, we have 
\begin{align*}
\int_{\partial (y_d+B)} g_n^2 d\sigma &\ \ =\ \left(2\prod_{j=2}^{n-3}I_j \right)\int_0^\pi \left(\frac{1}{a^{n-2}}-\frac{1}{R_d^{n-2}(\theta_1)}\right)^2 R_d^{n-2}(\theta_1) \sin^{n-2}\theta_1 \sqrt{R_d^2(\theta_1)+{R_d'}^2(\theta_1)}d\theta_1 \\
&\ \ =\ \left(2\prod_{j=2}^{n-3}I_j \right)\left(\frac{1}{a^{2n-4}} V_1^{n-2}(d)-\frac{2}{a^{n-2}}\big(I_{n-2}+V_2^{n-2}(d)\big)+V_3^{n-2}(d)\right)\\
&\ \ \ge\ \left(2\prod_{j=2}^{n-3}I_j \right)\left(\frac{1}{a^{2n-4}} V_1^{n-2}(0)-\frac{2}{a^{n-2}}\big(I_{n-2}+V_2^{n-2}(0)\big)+V_3^{n-2}(0)\right)\ =\ \int_{\partial B} g_n^2 d\sigma.
\end{align*}

Now take $n=2$. We use the following parameterization of the shifted sphere: $$y_d + \partial B = \left\{M(t)=\left(\sin{t},d+\cos{t}\right)\ |\ t\in [0,2\pi)\right\}.$$ 

Note that: $|M(t)| = 1+d^2+2 d \cos{t}$. We have 
\begin{align*}
\int_{\partial (y_d+B)} g_2^2 d\sigma &\ \ =\ \int_0^{2\pi} \left(\ln{\left(1+d^2+2 d \cos{t}\right)}-\ln{a}\right)^2dt \\
&\ \ =\ \int_0^{2\pi} \ln^2{\left(1+d^2+2 d \cos{t}\right)}dt -2\ln{a} \int_0^{2\pi} \ln{\left(1+d^2+2 d \cos{t}\right)}dt + 2\pi \ln^2{a} \\
&\ \ \ge\ 2\pi \ln^2{a} \ =\ \int_{\partial B} g_2^2 d\sigma,
\end{align*}
because $$\int_0^{2\pi} \ln^2{\left(1+d^2+2 d \cos{t}\right)}dt\ge 0 \ \ \ \ \text{and}\ \ \ \ \ \int_0^{2\pi} \ln{\left(1+d^2+2 d \cos{t}\right)}dt = 0.$$
Indeed, on the one hand the inequality is obvious and is an equality if and only if $d=0$, on the other hand the second assertion is a special case of a classical Lemma in complex analysis used in the proof of the so called Jensen formula (see for example  \cite[4.3.1]{alain}). We note that this equality can also be obtained by series expansion and the following classical identity (cf.  \cite[(6)]{catalan}):
\begin{equation}\label{cata}
\forall x \in \left(-\frac{1}{4},\frac{1}{4}\right),\ \ \ \ \sum_{n=1}^{+\infty} \frac{1}{n}\binom{2n}{n}x^n = 2\ln{\left(\frac{1-\sqrt{1-4x}}{2x}\right)}=2\ln{\left(\frac{2}{1+\sqrt{1-4x}}\right).}
\end{equation}
We have
\smaller\smaller
\begin{align*}
\int_0^{2\pi} \ln{\left(1+d^2+2 d \cos{t}\right)}d t &\ \ =\ \int_0^{2\pi} \ln{\left(1+d^2\right)}d t + \int_0^{2\pi} \ln{\left(1+\frac{2 d\cos{t}}{1+d^2}\right)}d t \\
&\ \ =\ 2\pi\ln{\left(1+d^2\right)} + \int_0^{2\pi} \sum_{n=1}^{+\infty} \frac{(-1)^{n+1}}{n} \left(\frac{2 d}{1+d^2}\right)^n\cos^n{t}\ d t \\
&\ \ =\ 2\pi\ln{\left(1+d^2\right)} +  \sum_{n=1}^{+\infty} \frac{(-1)^{n+1}}{n} \left(\frac{2 d}{1+d^2}\right)^n \int_0^{2\pi} \cos^n{t}\ d t \\
&\ \ =\ 2\pi\ln{\left(1+d^2\right)} - \pi \sum_{n=1}^{+\infty} \frac{1}{n}\binom{2n}{n} \left(\frac{d}{1+d^2}\right)^{2n}\ \ \text{(by \eqref{wallis_values}, $\int_0^{2\pi} \cos^{2n}{t}\ d t =2I_{2n} = \frac{2\pi}{2^{2n}}\binom{2n}{n}$ )}\\
&\ \ =\ 0\ \ \ \ \ \ \ \text{(we took $x=\left(\frac{d}{1+d^2}\right)^{2}$ in \eqref{cata})}. 
\end{align*} \larger \larger
It remains to prove that the last quantity is equal to zero. We recall the following classical identity:
$$\forall x\in(-1,1),\ \ \ \ \ \sum_{n=0}^{+\infty} I_n x^n = \frac{4}{\sqrt{1-x^2}}\arctan{\left(\sqrt{\frac{1+x}{1-x}}\right)}.$$ 
By writing the identity for $-x$ and adding each term together, we get
$$\forall x\in(-1,1),\ \ \ \ \ \sum_{n=0}^{+\infty} I_{2n} x^{2n} = \frac{2}{\sqrt{1-x^2}}\left(\arctan{\left(\sqrt{\frac{1+x}{1-x}}\right)}+\arctan{\left(\sqrt{\frac{1-x}{1+x}}\right)}\right)=\frac{\pi}{\sqrt{1-x^2}}.$$ 
Then for every $x\in(0,1)$,
\begin{align*}
\sum_{n=1}^{+\infty} \frac{1}{n}I_{2n} x^{2n} &\ \ =\ \sum_{n=1}^{+\infty} 2 I_{2n} \int_0^x u^{2n-1}d u = 2 \int_0^x \left(\sum_{n=1}^{+\infty}I_{2n} u^{2n-1}\right)d u = 2\pi \int_0^x \left(\frac{1}{u\sqrt{1-u^2}}-\frac{1}{u}\right)d u \\
&\ \ =2\pi \left[-\ln{\left(\sqrt{1-u^2}+1\right)}\right]^x_0= -2\pi \ln{\left(\frac{\sqrt{1-x^2}+1}{2}\right)}.\ 
\end{align*} 
By taking $x=\frac{2d}{1+d^2}$, we get $$\sum_{n=1}^{+\infty} \frac{1}{n}I_{2n} \left(\frac{2 d}{1+d^2}\right)^{2n} = 2\pi\ln{\left(1+d^2\right)},$$
which {completes} the proof.\vspace{2mm}

This completes the proof of the third assertion and the demonstration of Proposition \ref{propo}. 

\end{proof}
\subsection{Proof of Theorem \ref{main1}}
Finally, we conclude as before:
$$\tau_1(\Omega_d)\leq\frac{\int_{\Omega_d} |\nabla g_n|^2dx}{\int_{\partial (y_d+B)} g_n^2d\sigma}\leq\frac{\int_{\Omega_0} |\nabla g_n|^2dx}{\int_{\partial B} g_n^2d\sigma}=\tau_1(\Omega_0),$$
with equality if and only if $d=0$. This ends the proof of Theorem \ref{main1}.




\section{Computation of the first Steklov eigenvalue of spherical shells}\label{s:appendix}
In the present section, we compute the Steklov eigenvalues of the spherical shell $\Omega_0=B\backslash a B\subset \mathbb{R}^n$, where $a\in (0,1)$. We then prove a monotonicity result on these eigenvalues, which allows us to give the exact value of $\sigma_1(\Omega_0)$ and its corresponding eigenfunctions. 

\begin{theorem}\label{steklov_eigenvalue}
Let $n\ge2$. The first nonzero Steklov eigenvalue of the spherical shell $\Omega_0=B\backslash a B\subset \mathbb{R}^n$ is
$$\sigma_1(\Omega_0) = \frac{(n+1)a^{n+1}+a^n+a+n-1-\sqrt{\left((n+1)a^{n+1}+a^n+a+n-1\right)^2-4(n-1)a\left(1-a^n\right)^2}}{2a\left(1-a^n\right)}\cdot$$
It is of multiplicity $n$ and the corresponding eigenfunctions are
$$\begin{array}{ccccc}
u^i_n & : & \mathbb{R}^n & \longrightarrow & \mathbb{R} \\
 & & x=(x_1,\cdots,x_n) & \longmapsto & x_i\left(1+\frac{\mu_{\sigma,n}}{|x|^n}\right), \\
\end{array}$$
where $i\in \llbracket 1,n\rrbracket$ and $\mu_{\sigma,n} = \frac{1-\sigma_1(\Omega_0)}{n+\sigma_1(\Omega_0)-1}$.
\end{theorem}

\begin{remark}
Theorem \ref{steklov_eigenvalue} has already been proved for the planar case by B. Dittmar \cite{MR1421478} (see also \cite{phd}). For higher dimensions, A. Fraser and R. Schoen \cite{fraser} gave asymptotic formula for the lowest eigenvalues of spherical shells when the hole is vanishing. In this case, it is easy to identify the first eigenvalues (in particular the first one). Unfortunately, this is no longer the case when the hole is not vanishing as explained in sections \ref{c} and \ref{w}. 
\end{remark}

\subsection{Computation of the eigenvalues via classical separation of variables technique}\label{c}\vspace{4px}


Finding the eigenvalues and eigenfunctions of the Laplacian on special domains (balls, rectangles, annulus...) is a classical problem (see for example \cite[Section 3]{MR3124880}). The standard method is to look for eigenfunctions via separation of variables and then prove that they form a complete basis of a convenient function space, this combined with orthogonality properties of the eigenfunctions shows that we didn't miss any {eigenvalues and eigenfunctions}.\vspace{4px}

Take $k\in\mathbb{N}$, let us search harmonic functions $h_k$ of the form 
$$\begin{array}{ccccc}
h_k & : & \mathbb{R}_+\times[0,\pi]\times\cdots\times [0,\pi]\times [0,2\pi] & \longrightarrow & \mathbb{R} \\
 & & (r,\theta_1,\cdots,\theta_{n-1}) & \longmapsto & \alpha_k(r)\beta_k(\theta_1,\cdots,\theta_{n-1}), \\
\end{array}$$
where $\beta_k\in H_k^n$ is a spherical harmonic of order $k$ and $H_k^n$ is the set of restrictions of homogeneous harmonic polynomial of degree $k$ with $n$ variables on the unit sphere $\partial B$ (for an introduction to harmonic polynomials we refer to \cite[Chapter 5]{harmonic}). It is well-known that the set $H_k^n$ corresponds to the eigenspace of the Laplace-Beltrami operator $-\Delta_{\partial B}$ associated to the eigenvalue {$k(k+n-2)$}. \vspace{4px}

We have
{
$$\Delta h_k = \left(\frac{\partial^2}{\partial r^2}+\frac{n-1}{r}\frac{\partial }{\partial r}+r^{-2}\Delta_{\partial B}\right)h_k = \left(\alpha_k''(r)+\frac{n-1}{r}\alpha_k'(r)-\frac{k(k+n-2)}{r^2}\alpha_k(r)\right)\beta_k(\theta_1,\cdots,\theta_{n-1}).$$}

The condition $\Delta h_k = 0$ implies that $\alpha_k$ must satisfy the differential equation {$$\alpha_k''(r)+\frac{n-1}{r}\alpha_k'(r)-\frac{k(k+n-2)}{r^2}\alpha_k(r)=0.$$}

By standard methods of solving ODEs, the solutions of the last equation are given by
$$\alpha_0(r)=\left\{\begin{matrix}  p_{0,2} + q_{0,2} \ln{r}\ \ \ \ \text{if $n=2$,}
\vspace{4px}
\\
p_{0,n} + \frac{q_{0,n}}{r^{n-2}} \ \ \ \ \ \ \ \ \ \ \text{if $n\ge3$,}
\end{matrix}\right.$$
and for $k\ge 1$,
$$\alpha_k(r) = p_{k,n}r^k + \frac{q_{k,n}}{r^{k+n-2}},$$
where $p_{k,n}$ and $q_{k,n}$ are constants.

It remains to look for all possible values $\delta_k$ such that $\frac{\partial h_k}{\partial n} = \delta_k h_k$ on $\partial \Omega_0$. This equality is equivalent to 
$$\left\{
\begin{array}{l}
  \alpha_k'(1)=\delta_k \alpha_k(1),  \vspace{4px}\\
  \alpha_k'(a)=-\delta_k \alpha_k(a).  \\
\end{array}
\right.$$

As explained in the proof of  \cite[Proposition 3]{fraser}, those equalities imply that the possible eigenvalues $\delta_k$ are solutions of equations of second order.\vspace{4px} 

When $k=0$, we find two eigenvalues: $0$ that corresponds to constant eigenfunctions and $\delta_0$ that corresponds to a (non-constant) radial one. 

$$\delta_0=\left\{\begin{matrix}  \frac{1+a}{a\ln{1/a}},\ \ \ \ \ \ \ \ \ \ \ \ \ \ \ \ \ \ \text{if $n=2$,}
\vspace{5px}
\\
\frac{(n-2)(1+a^{n-1})}{a(1-a^{n-2})}, \ \ \ \ \ \ \ \ \ \ \text{if $n\ge3$.}
\end{matrix}\right.$$

The corresponding (radial) eigenfunction is given by:
$$h_0(r,\theta_1,\cdots,\theta_{n-1})=\left\{\begin{matrix}   1+\delta_0\ln{r},\ \ \ \ \ \ \ \ \ \ \ \ \ \ \ \ \ \ \text{if $n=2$,}\vspace{3px}
\\
(2-n-\delta_0)+\frac{\delta_0}{r^{n-2}}, \ \ \ \ \ \ \ \text{if $n\ge3$.}
\end{matrix}\right.$$

On the other hand, (as mentioned in \cite{fraser}) when $k\ge 1$, one finds two eigenvalues $\delta_k^{(1)}<\ \delta_k^{(2)}$ corresponding to the solutions of the following equation:
\begin{equation}\label{equation}
 A_k\delta^2 + B_k\delta +C_k = 0,
\end{equation}
where
$$\left\{
\begin{array}{l}
  A_k = a-a^{2k+n-1},\vspace{4px}
  \\
  B_k = -\big((k+n-2)a^{2k+n-1} + ka^{2k+n-2}+ka +k+n-2\big), \vspace{4px}
  \\
  C_k = (k+n-2)k(1-a^{2k+n-2}).  \\
\end{array}
\right.$$

We compute the determinant $\Delta_k$, and use the fact that $a\in (0,1)$ to check that $\Delta_k>0$.
{
\begin{align*}
\Delta_k &\ \ =\ B_k^2-4A_k C_k \\
&\ \ =\ \Big[ (k+n-2)a^{2k+n-1}+k a^{2k+n-2}+ka+k+n-2 \Big]^2 - 4(k+n-2)ka(1-a^{2k+n-2})^2 \\
&\ \ \ge\ (ka+k+n-2)^2-4(k+n-2)k a\big(1-a^{2k+n-2}\big)^2  \ \ \ \ \ \text{(because $(k+n-2)a^{2k+n-1}+k a^{2k+n-2}\ge0$)}  \\
&\ \ \ge\ \big((k+n-2) + k a\big)^2-4(k+n-2)k a \ \ \ \ \ \text{(because $0\leq (1-a^{2k+n-2})^2<1$)}  \\
&\ \ =\ \big((k+n-2)-k a\big)^2 >0. 
\end{align*}} 

Then, the equation \eqref{equation} admits two different positive solutions $\delta_k^{(1)}:=\frac{-B_k-\sqrt{\Delta_k}}{2A_k}\ <\ \frac{-B_k+\sqrt{\Delta_k}}{2A_k}=: \delta_k^{(2)}$.\vspace{4px}

By straightforward computations, the corresponding eigenfunctions are given by: 
\begin{equation}\label{eigenfunctions}
h_k^{(i)}(r,\theta_1,\cdots,\theta_{n-1}) = \left(r^k+\frac{k-\delta_k^{(i)}}{n+\delta_k^{(i)}+k-2}\cdot\frac{1}{r^{k+n-2}}\right) Y_{k,j}(\theta_1,\cdots,\theta_{n-1}),
\end{equation}
where $Y_{k,j}\in H_k^n$ corresponds to the $j$-th (with $j\in \llbracket 1,\text{dim}\  H_k^n\rrbracket$) spherical harmonic of order $k$  and $i\in\{1,2\}$. 

Thus, the multiplicity of $\delta_k^{(i)}$ is equal to $$\text{dim}\  H_k^n\ =\ \begin{pmatrix}n+k-1
\\ 
n-1
\end{pmatrix}-\begin{pmatrix}n+k-3
\\ 
n-1
\end{pmatrix}.$$ 

At last, by using expansions results for harmonic functions on annuli (see \cite[Section 9.17]{harmonic}, for $n=2$ and \cite[Section 10.1]{harmonic}, for $n\ge 3$), we deduce that the eigenfunctions {we found, form a complete basis} of the space of harmonic functions on the annulus $\Omega_0$. 

It remains to determine the lowest eigenvalue between $\delta_0$ and the $\delta_k^{(i)}$ for $k\in \mathbb{N}^*$ and $i\in \{1,2\}$. 
\subsection{A monotonicity result\vspace{4px}}\label{w}
We state and prove the following key lemma, which combined with results of Section \ref{c} gives an immediate proof of Theorem \ref{steklov_eigenvalue}.
\begin{lemma}\label{monotonicity} We have:
\begin{enumerate}
    \item The sequence $\big(\delta_k^{(1)}\big)_{k\ge1}$ is strictly increasing.
        \item $\sigma_1(\Omega_0)<\delta_0$.
\end{enumerate}
\end{lemma}
\begin{proof}
The case $n=2$ had been considered in \cite{MR1421478,phd}. Let $n\ge3$, 
\begin{enumerate}
    \item 
    we have
    $$\delta_k^{(1)} = \frac{2C_k}{-B_k+\sqrt{B_k^2-4A_k\times C_k}} = \frac{2(k+n-2)k(1-a^{2k+n-2})}{-B_k+\sqrt{B_k^2-4(k+n-2)ka(1-a^{2k+n-2})^2}}\cdot$$
    The idea of proof is to write $\delta_k^{(1)} = P_k/Q_k$, where $(P_k)_k$ (resp. $(Q_k)_k$) is a positive increasing (resp. decreasing) sequence. Indeed, we can write
    $$\delta_k^{(1)} = \frac{2\sqrt{(k+n-2)k}(1-a^{2k+n-2})}{-\frac{B_k}{\sqrt{(k+n-2)k}}+\sqrt{\left(-\frac{B_k}{\sqrt{(k+n-2)k}}\right)^2-4a(1-a^{2k+n-2})^2}}\cdot$$

    The sequences $\left(2\sqrt{(k+n-2)k}(1-a^{2k+n-2})\right)_k$ and $\left(a(1-a^{2k+n-2})\right)_k$ are strictly increasing. It remains to prove that the (positive) sequence $\Big(-\frac{B_k}{\sqrt{(k+n-2)k}}\Big)_{k\ge 1}$ is strictly decreasing. 
   
    We have 
 \begin{align*}
-\frac{B_k}{\sqrt{(k+n-2)k}} &\ \ =\ \frac{(k+n-2)a^{2k+n-1}+ka^{2k+n-2}+ka+k+n-2}{\sqrt{k}\sqrt{k+n-2}}\\
&\ \ =\ \frac{(k+n-2)\big(a^{2k+n-1}+1\big)+k a\big(a^{2k+n-3}+1\big)}{\sqrt{k}\sqrt{k+n-2}}\\
&\ \ =\ \sqrt{\frac{k+n-2}{k}}\big(a^{2k+n-1}+1\big)+a\sqrt{\frac{k}{k+n-2}}\big(a^{2k+n-3}+1\big).
\end{align*}  
Let us introduce the function: 
$$\begin{array}{ccccc}
h_{a,n} & : & [1,+\infty[ & \longrightarrow & \mathbb{R} \\
 & & t & \longmapsto & \sqrt{\frac{t+n-2}{t}}\big(a^{2t+n-1}+1\big)+a\sqrt{\frac{t}{t+n-2}}\big(a^{2t+n-3}+1\big). \\
\end{array}$$

we prove that $h_{a,n}$ is strictly decreasing. To do so, we compute the derivative $h'_{a,n}$ and prove that it is negative on $[1,+\infty[$.

{We have for every $t\ge1$,
 \begin{align*}
h'_{a,n}(t)&\ \ =\ \frac{-\frac{n-2}{t^2} \left(a^{n+2 t-1}+1\right)}{2 \sqrt{\frac{n+t-2}{t}}}+\frac{\frac{n-2}{(n+t-2)^2}a \left(a^{n+2 t-3}+1\right)}{2 \sqrt{\frac{t}{n+t-2}}}\\
&\ \ +\ 2 \ln (a) \sqrt{\frac{t}{n+t-2}} a^{n+2 t-2}+2 \ln (a) \sqrt{\frac{n+t-2}{t}} a^{n+2 t-1}\\
&\ \ <\ \frac{-\frac{n-2}{t^2} \left(a^{n+2 t-1}+1\right)}{2 \sqrt{\frac{n+t-2}{t}}}+\frac{\frac{n-2}{(n+t-2)^2}a \left(a^{n+2 t-3}+1\right)}{2 \sqrt{\frac{t}{n+t-2}}} \ \ \ \ \text{(because $\ln(a)<0$).}\\
&\ \ =\ \frac{(n-2)(a^{n+2t-1}+1)}{2t(n+t-2)}\sqrt{\frac{t}{n+t-2}}\cdot \Big(a\cdot\frac{a^{n+2t-3}+1}{a^{n+2t-1}+1}-1-\frac{n-2}{t}\Big) \\
&\ \ <\ \frac{(n-2)(a^{n+2t-1}+1)}{2t(n+t-2)}\sqrt{\frac{t}{n+t-2}}\cdot \Big(\frac{a^{n+2t-3}+1}{a^{n+2t-1}+1}-\frac{t+n-2}{t}\Big)\ \ \ \ \text{(because $a\in (0,1)$).}
\end{align*}} 
We have 
$$\frac{a^{n+2t-3}+1}{a^{n+2t-1}+1}-\frac{t+n-2}{t}<0 \ \ \Leftrightarrow\ \ \ \frac{t+n-2}{t} a^{n+2t-1}- a^{n+2t-3}+\frac{n-2}{t}>0.$$
Now, let $t\ge1$ and $n\ge 3$. We consider the function 
$$\begin{array}{ccccc}
g_{t,n} & : & (0,1) & \longrightarrow & \mathbb{R} \\
 & & a & \longmapsto & \frac{n+t-2}{t} a^{n+2t-1}- a^{n+2t-3}+\frac{n-2}{t}\cdot \\
\end{array}$$
We compute the derivative of $g_{t,n}$ on $(0,1)$. We have for every $a\in (0,1)$, 
$$g'_{t,n}(a) = \frac{n+t-2}{t}(n+2t-1) a^{n+2t-4}\Big(a^2-\frac{t}{n+t-2}\cdot\frac{n+2t-3}{n+2t-1}\Big).$$
We deduce that $g_{t,n}$ is decreasing on $(0,a_{t,n})$ and increasing on $(a_{t,n},1)$, which implies that it attains its minimum in $a_{t,n}$, where 
$$a_{t,n} = \sqrt{\frac{t}{n+t-2}\cdot\frac{n+2t-3}{n+2t-1}}\cdot$$
We have
 \begin{align*}
g_{t,n}(a) &\ \ \ge\ \ g_{t,n}(a_{t,n})\\
&\ \ =\ \frac{n+t-2}{t}\Big(\frac{t}{n+t-2}\Big)^{\frac{n+2t-1}{2}}\Big(\frac{n+2t-3}{n+2t-1}\Big)^{\frac{n+2t-1}{2}}-\Big(\frac{t}{n+t-2}\Big)^{\frac{n+2t-3}{2}}\Big(\frac{n+2t-3}{n+2t-1}\Big)^{\frac{n+2t-3}{2}}+\frac{n-2}{t} \\
&\ \ =\ -\Big(\frac{t}{n+t-2}\Big)^{\frac{n+2t-3}{2}}\Big(\frac{n+2t-3}{n+2t-1}\Big)^{\frac{n+2t-3}{2}} \frac{2}{n+2t-1}+\frac{n-2}{t}\\
&\ \ \ge\ -\frac{1}{t+\frac{n-1}{2}} +\frac{n-2}{t}\ >\ 0\ \ \ \ \ \text{(because $n-2\ge1$ and $t+\frac{n-1}{2}>t$).}
\end{align*}  
We deduce that for all $t\ge 1$: $h'_{a,n}(t)<0$, which implies that $h_{a,n}$ is strictly decreasing on $[1,+\infty[$. In particular, the sequence $\Big(-\frac{B_k}{\sqrt{(k+n-2)k}}\Big)_{k\ge 1}$ is strictly decreasing and so is $$\left(\ -\frac{B_k}{\sqrt{(k+n-2)k}}+\sqrt{\left(-\frac{B_k}{\sqrt{(k+n-2)k}}\right)^2-4a(1-a^{2k+n-2})^2}\ \right)_{k\ge 1}.$$
\item Take $\gamma: x\in\mathbb{R}^n\longmapsto x_1$ an eigenfunction corresponding to the first nonzero Steklov eigenvalue of the unit ball $B$ centred at the origin $O$. This function can be used as a test function in the variational definition of $\sigma_1(B\backslash a B)$. 

We write 
 \begin{align*}
\sigma_1(B\backslash a B) &\ \ =\ \ \inf\left\{ \frac{\int_{B\backslash a B} |\nabla u|^2dx}{\int_{\partial (B\backslash a B)} u^2 d\sigma}\ \ \Big|\ \ u\in H^1(\Omega)\backslash \{0\}\ \text{such that}\ \int_{\partial \Omega} u d \sigma= 0  \right\}\\
&\ \ \leq\ \frac{\int_{B\backslash a B} |\nabla \gamma|^2dx}{\int_{\partial B\cup \partial (a B)} \gamma^2 d\sigma}\\
&\ \ \leq\  \frac{\int_{B} |\nabla \gamma|^2dx}{\int_{\partial B} \gamma^2 d\sigma}\\
&\ \ =\  \sigma_1(B)\\
&\ \ =\  1\ \ \ \ \ \ \ \ \ \ \ \   \text{(see \cite[Example 1.3.2]{MR3662010})}\\
&\ \ <\ (n-2)\frac{1+a^{n-1}}{a(1-a^{n-2})}\\
&\ \ =\  \delta_0.
\end{align*} 
\end{enumerate}
\end{proof}
\subsection{Proof of Theorem \ref{steklov_eigenvalue}}
We have $\delta_1^{(1)}<\delta_1^{(2)}$ and  by Lemma \ref{monotonicity} 
$$\sigma_1(\Omega_0)<\ \delta_0\ \ \text{and}\ \ \forall k\ge2,\ \ \ \ \ \delta_1^{(1)}<\ \delta_k^{(1)}<\delta_k^{(2)}.$$
This implies that $\delta_1^{(1)}$ is the lowest nonzero Steklov eigenvalue of $\Omega_0$, which writes
$\sigma_1(\Omega_0)=\delta_1^{(1)}$. It is of multiplicity $n$ {and the corresponding eigenfunctions, given by \eqref{eigenfunctions}, are as follows:}
$$\begin{array}{ccccc}
u^i_n & : & \mathbb{R}^n & \longrightarrow & \mathbb{R} \\
 & & x=(x_1,\cdots,x_n) & \longmapsto & \left(|x|+\frac{\mu_{\sigma,n}}{|x|^{n-1}}\right)\frac{x_i}{|x|}\ =\ x_i\left(1+\frac{\mu_{\sigma,n}}{|x|^n}\right),   \\
\end{array}$$
where $i\in \llbracket 1,n\rrbracket$ and $\mu_{\sigma,n} = \frac{1-\sigma_1(\Omega_0)}{n+\sigma_1(\Omega_0)-1}$.

\section*{Acknowledgments}
The author would like to thank Antoine Henrot for pointing out this problem and Mamoune Benchekroun, Antoine Henrot and Jimmy Lamboley for useful discussions.\vspace{3px}

The author would also like to thank the anonymous referees for their careful reading and usefull comments that helped to improve the manuscript.\vspace{3px}

This work was partially supported by the project ANR-18-CE40-0013 SHAPO financed by the French Agence Nationale de la Recherche (ANR).






\bibliographystyle{plain}
\bibliography{sample.bib}

\end{document}